\documentclass[a4paper]{amsart}

\usepackage{amsthm}
\usepackage{amsmath}
\usepackage{amssymb}
\usepackage{latexsym}
\usepackage{enumerate}
\usepackage{graphicx}% Include figure files
\usepackage[utf8]{inputenc}
\usepackage{epsfig} %inclosure fig EPS
\usepackage{pgf}
\usepackage{tikz}
\usepackage{float}
\usepackage{dsfont}
\usepackage{times}
\everymath{\displaystyle}

\renewcommand{\thetheoremName}

\newtheorem*{thmain}{Main Theorem}

\newtheorem*{open}{Open question}

\newtheorem{proposition[[]]}[theoremName]{Proposition G}

\newtheorem{theorem}{Theorem}[section]

\newtheorem{proposition}[theorem]{Proposition}
\newtheorem{corollary}[theorem]{Corollary}

\theoremstyle{definition}

\newtheorem{remark}{Remark}

\numberwithin{equation}{section}

\newcommand{\Vol}{\operatorname{Vol}}

\newcommand{\erre}{\mathbb{R}}
\newcommand{\He}{\mathcal{H}}

\newcommand{\Q}{\mathcal{Q}}
\newcommand{\E}{\mathcal{E}}

\begin{document}

\title[Large time behavior of the on-diagonal heat kernel]{Large time behavior of the on-diagonal heat kernel for minimal submanifolds with polynomial volume growth}

\author{Vicent Gimeno}      
\address{Department of Mathematics-INIT, Universitat Jaume I, Castell\'o de la Plana, Spain                        %  \\
%             \emph{Present address:} of F. Author  %  if needed
}
\email{gimenov@uji.es}

\thanks{Work partially supported by DGI grant MTM2010-21206-C02-02.}

%\authorrunning{Short form of author list} % if too long for running head

\begin{abstract}
In this paper we provide a lower bound for the long time on-diagonal heat kernel of minimal submanifolds in a Cartan-hadamard ambient manifold assuming that the submanifold is of polynomial volume growth. In particular cases, that lower bound is  related with the number of ends of the submanifold. 
\keywords{heat kernel \and minimal submanifold \and Cartan-Hadamard \and volume growth \and number of ends}
 \subjclass{35P15}
\end{abstract}

\maketitle

\section{Introduction}\label{intro}
Let $M^m$ be a $m$-dimensional minimally immersed submanifold into a simply connected  ambient manifold $N^n$ with sectional curvatures $K_N$ bounded from above by $K_N\leq 0$. S. Markvorsen proved in \cite{Mar} -following \cite{Ch-Li-Yau}-  that  the  heat kernel $\He$ of $M^m$ is bounded from above by the heat kernel $\He^{m,0}$ of the Euclidean space $\erre^m$, namely:
\begin{equation}\label{Mar}
\He(t,x,y)\leq \He^{m,0}(t,r_x(y))=\frac{1}{\left(4\pi t\right)^{\frac{m}{2}}}e^{-\frac{\left(r_x(y)\right)^2}{4t}},
\end{equation}
being $r_x(y)$ the distance in $N$ from $x$ to $y$.
In particular for the on-diagonal heat kernel $\He(t,x,x)$ of $M^m$ one can state that 
\begin{equation}\label{upper}
\left(4\pi t\right)^\frac{m}{2}\He(t,x,x)\leq 1.
\end{equation}

This paper deals with lower bounds to the on-diagonal heat kernel assuming certain restriction on the volume growth. In  order to define that appropriate behavior on the growth of the extrinsic volume, recall that given a minimal submanifold $M^m$ properly immersed in a Cartan-Hadamard manifold $N$ with sectional curvatures $K_N$ bounded from above by $K_N\leq 0$ and denoting by $\omega_m$ the volume of a radius one geodesic ball in $\erre^m$ and by $B_R^N(p)$ the geodesic ball in $N$ of radius $R$ centered at $p$, by the monotonicity formula (see for instance \cite[theorem 2.6.9]{MP-2012} and \cite{Palmer}) for any point $p\in M^m$  the function 
\begin{equation}\label{monotonicity}
\Q(R)=\frac{\Vol(M^m\cap B_R^N(p))}{\omega_m R^m},
\end{equation}
is a non decreasing function. Throughout this paper  a complete minimal submanifold properly immersed in a Cartan-hadamard ambient manifold is called a minimal submanifold of \emph{polynomial volume growth} if there exists a constant $\E$ depending on $M^m$ such that:
\begin{equation}\label{def-poly}
\lim_{R\to\infty}\Q(R)\leq \E<\infty.
\end{equation}
Under such volume growth behavior we can state the behavior of the long time asymptotic for the on-diagonal heat kernel by the main theorem of this paper. The main theorem makes use of the following constant $C_m$ depending only on the dimension $m$ of the submanifold 
\begin{equation}
C_m:=\frac{\Gamma\left(\frac{m}{2},2 \left(\frac{m}{2}\Gamma\left(\frac{m}{2}\right)\right)^{\frac{2}{m}}\right)}{\Gamma(\frac{m}{2})},
\end{equation}
where $\Gamma(z)$ and 
$\Gamma(z_1,z_2)$ in the above expression denote the gamma function and the incomplete gamma function respectively, i.e,
$$
\begin{aligned}
\Gamma(z):=&\int_{0}^\infty t^{z-1}e^{-t}dt.\\
\Gamma(z_1,z_2):=&\int_{z_2}^\infty t^{z_1-1}e^{-t}dt.
\end{aligned}
$$
For minimal submanifolds with an extrinsic volume growth controlled by the above constant $C_m$ we can state the main result of this paper:
\begin{thmain}
Let $M^m$ be a complete $m$-dimensional submanifold properly immersed in a simply connected  ambient manifold $N$ with sectional curvatures $K_N$ bounded from above by $K_N\leq 0$. Suppose that $M^m$ is of polynomial volume growth, and that
\begin{equation}
\E<\frac{1}{C_m},
\end{equation}
Then, the heat kernel $\He$ of $M^m$ satisfies
\begin{equation}\label{result}
\frac{\left(1-\E C_m\right)^2}{\E}\leq\limsup_{t\to\infty}\left(4\pi t\right)^\frac{m}{2}\He(t,x,x)\leq 1.
\end{equation} 
\end{thmain}

\begin{figure}
\begin{center}
\includegraphics[width=30mm]{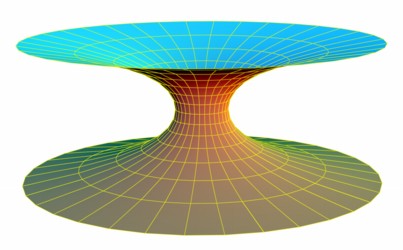} \quad \includegraphics[width=30mm]{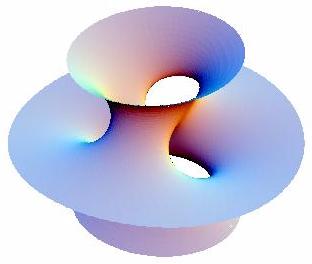} \quad \includegraphics[width=30mm]{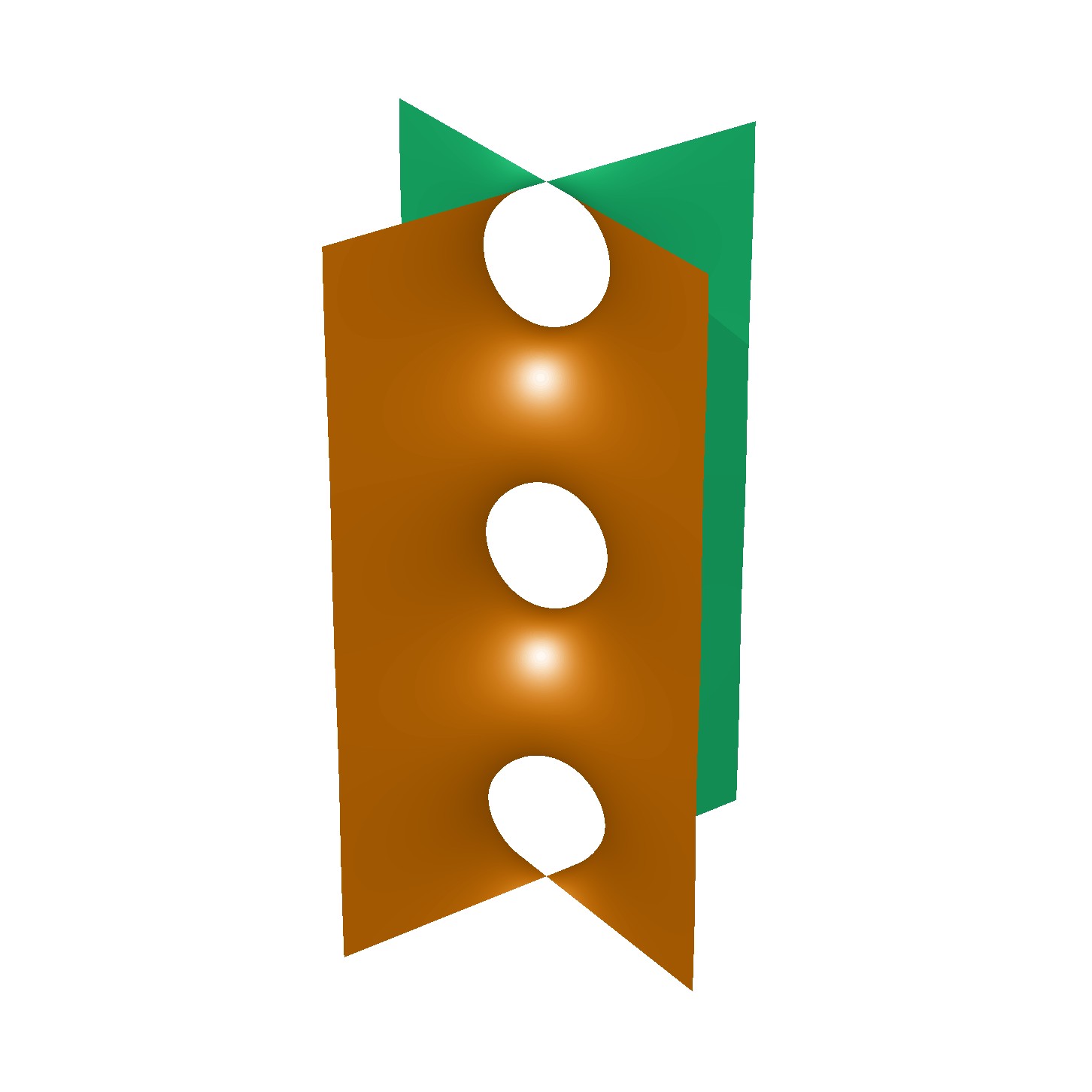}
\end{center}
\caption{The catenoid, the Costa surface and the Scherk singly periodic surface are examples of minimal surfaces immersed in $\erre^3$ with polynomial volume growth which is equivalent to quadratic area growth when the submanifold is a surface. }
\label{example}
\end{figure}

 It is not hard to find examples of complete minimal submanifolds properly and minimally immersed in a Cartan-Hadamard ambient manifold with polynomial volume growth. Indeed, for a complete minimal surface embedded in $\erre^3$, by a well known result  (see \cite{Oss,JM} and introduction in \cite{GP}), if the surface has finite total curvature then the surface has polynomial volume growth (quadratic area growth) and the constant $\E$ given in equation (\ref{def-poly}) is equal to the number of ends of the surface. This is the case of the catenoid or the Costa surface (with $\E=2$ for the catenoid and $\E=3$ for the Costa surface). It is also known  that there exist other surfaces with quadratic area growth but without finite total curvature and even without finite topological type. An example of that kind of surface is the Scherk singly periodic surface (see introduction in \cite{Meeks2007}) which has $\E=2$.

Since 
$$
C_2\sim 0.14 \quad \frac{1}{C_2}\sim 7.39, 
$$  
we can apply the main theorem to the catenoid, the Costa and the Scherk surface, obtaining
$$
\frac{\left(1-0.28\right)^2}{2}\leq \limsup_{t\to\infty}\left(4\pi t\right)\He(t,x,x)\leq 1,
$$
for the catenoid and the Scherk singly periodic surface, and
$$
\frac{\left(1-0.41\right)^2}{3}\leq \limsup_{t\to\infty}\left(4\pi t\right)\He(t,x,x)\leq 1,
$$
for the Costa surface.

As we have shown,  there are several examples where the volume growth is related with the number of ends of the submanifold.  In fact, the following theorem allow us to achieve inequality (\ref{def-poly}) under certain decay of the norm of the second fundamental form and to give a topological meaning to $\lim_{R\to\infty}\Q(R)$

\begin{theorem}[see  theorem 2.2 of \cite{Che4} and \cite{GPGap}] Let $M^m$ be an $m-$dimensional complete immersed minimal submanifold in $\erre^n$ which satisfies
\begin{equation}
\lim_{R\to\infty}\underset{r(x)\geq R}{\sup_{x\in M^m}}r(x)\Vert A\Vert (x)=0,
\end{equation}
where $A$ denotes the second fundamental form. Then, the number of ends $\E\left(M^m\right)$ of $M^m$ is given by:
\begin{equation}
\lim_{R\to\infty}\Q(R)=\E(M^m)
\end{equation}
provided either of the following two conditions is satisfied:
\begin{enumerate}
\item $m=2$, $n=3$ and each end of $M^m$ is embedded.
\item $m\geq 3$. 
\end{enumerate}
\end{theorem}

Hence, we can state the following corollary showing the relation between the number of ends and the lower bound for the heat kernel under the assumptions of the above theorem (see introduction of \cite{GriEnds} for a complete overview on the two sides estimates for the heat kernel on manifolds with ends):

\begin{corollary}
Let $M^m$ be an $m-$dimensional complete immersed minimal submanifold in $\erre^n$ which satisfies
\begin{equation}\label{decay}
\lim_{R\to\infty}\underset{r(x)\geq R}{\sup_{x\in M^m}}r(x)\Vert A\Vert (x)=0,
\end{equation}
and 
\begin{enumerate}
\item if  $m=2$ and $n=3$, each end of $M^m$ is embedded. Or,
\item $m\geq 3$. 
\end{enumerate}
Then, if the number of ends $\E(M^m)$ of $M^m$ is bounded from above by
\begin{equation}
\E(M^m)<\frac{1}{C_m},
\end{equation}
the heat kernel $\He$ of $M^m$ satisfies
\begin{equation}
\frac{\left(1-\E(M^m) C_m\right)^2}{\E(M^m)}\leq\limsup_{t\to\infty}\left(4\pi t\right)^\frac{m}{2}\He(t,x,x)\leq 1.
\end{equation} 
\end{corollary}

If $M^2$ is a minimal surface in $\erre^3$, by the Gauss formula the second fundamental form is related with the Gaussian curvature $K_G$ of $M ^2$ by
\begin{equation}\label{Gauss}
K_G =- \frac{1}{2}\vert A \vert^2,
\end{equation}
in view of \cite[theorem 1.2]{MPR-removable} it seems that in the particular case of complete embedded minimal surfaces in $\erre^3$ if there exists a constant $C$ such that $\vert K_G\vert R^2\leq C$, then:
$$
\vert K_G\vert R^2\leq C\quad \rightarrow \int_{M^2}\vert K_G\vert <\infty \rightarrow  \lim_{R\to\infty}\underset{r(x)\geq R}{\sup_{x\in M^m}}r(x)\vert A\vert (x)=0.
$$
Hence, the condition given in equation (\ref{decay}) in the above corollary can be replaced in the particular case of complete embedded minimal surfaces in $\erre^3$ by
$$
\vert K_G\vert R^2\leq C.
$$
Recall also that a particular case when equality (\ref{decay}) holds is (see \cite{Che4}) when
$$
\int_{M^m}\vert A \vert^mdV<\infty
$$
i.e,. when the submanifold has finite scalar curvature (see also \cite{A1}).

Let us finally remark that

\begin{remark}
Given a manifold $M^n$ with non-negative Ricci  curvature $\text{Rc}>0$, Bishop-Gromov volume comparison theorem asserts that for any $o\in M^n$ the relative volume quotient $\frac{\Vol(B_R^{M^n}(o))}{\omega_nR^n}$ is non-increasing in the radius $R$ (being $B_R^{M^n}(o)$ the geodesic ball of radius $R$ centered at $o$). The relative volume quotient converges to  a non-negative number $\Theta$:
$$
\lim_{R\to \infty}\frac{\Vol(B_R^{M^n}(o))}{\omega_nR^n}=\Theta\geq 0.
$$ 
If $\Theta>0$, one says that the manifold $M^n$ has \emph{maximal volume growth}. 

P. Li  proved in \cite{Li86} (see also \cite{Xu2013}) that if $M^n$ has $\text{Rc}>0$ and maximal volume growth, then
\begin{equation}
\lim_{t\to\infty}\Vol\left(B^{M^n}_{\sqrt{t}}\left(y\right)\right)\He\left(t,x,y\right)=\omega_n\left(4\pi\right)^{-\frac{n}{2}}.
\end{equation}
Therefore
\begin{equation}\label{Li-result}
\lim_{t\to\infty}\left(4\pi t\right)^{\frac{n}{2}}\He(t,x,y)=\frac{1}{\Theta}.
\end{equation}

In some sense, our main theorem can be understood (partially) as a reverse of the Li's theorem because at least on dimension $2$, by the Gauss formula (equation (\ref{Gauss})), a submanifold properly and minimally immersed in a Cartan-Hadamard ambient manifold has non-positive sectional curvature (instead of $\text{Rc}>0$) and  because, by the monotonicity formula, the extrinsic quotient given in equation (\ref{monotonicity}) is non-decreasing (instead of non-increasing like the relative volume quotient). 

Despite of the weakness of the inequalities (\ref{result}) in comparison to equality (\ref{Li-result}) observe, however, that a non-negatively Ricci-curved manifold with maximal volume growth must have finite fundamental group (see \cite{Li86}) but that is not true for minimal submanifolds of a Cartan-Hadamard with polynomial volume growth (see for instance the singly periodic Scherk surface (figure \ref{example})). 
\end{remark}

The most well known examples of heat kernels of minimal submanifolds $M^m$ in the Euclidean space $\erre^n$ are when $M^m$ is a totally geodesic  submanifold $\erre^m$ in $\erre^n$. Observe that in that case $\E=1$ if $C_m$ were $0$ the inequality (\ref{result}) would be an exact equality. Therefore, it is a natural question to ask the following open question

\begin{open}
 Is it possible to improve the main theorem changing $C_m$ by $0$?
\end{open}

The structure of the paper is as follows

In \S 2 we recall the definition and several properties of the heat kernel on a Riemannian manifold and provide proposition \ref{completeness} which states that every complete minimal submanifold with polynomial volume growth is stochastically complete. With those previous requirements we can, in \S 3, to prove the main theorem.
\section{Preliminaries}
Let $M$ be a Riemannian manifold with (possibly empty) smooth boundary $\partial M$, and denote by $\Delta$ the Laplacian on $M$. The heat kernel on $M$ is a function $\He(t,x,y)$ on $(0,\infty)\times M\times M$ which is the minimal positive fundamental solution to the heat equation
\begin{equation}
\frac{\partial v}{\partial t}=\Delta v\quad.
\end{equation} 

In other words, the Cauchy problem with Dirichlet boundary conditions
\begin{equation}
\begin{cases}
\frac{\partial v}{\partial t}=\Delta v\quad,\\
v\vert_{t=0}=v_0(x)\quad,
\end{cases}
\end{equation}
has a solution 
\begin{equation}
v(x,t)=\int_{M}\He(t,x,y)v_0(y)d\mu_y\quad,
\end{equation}
provided that $v_0$ is a bounded continuous positive function. 
Moreover the heat kernel has the following properties:
\begin{enumerate}
\item Symmetry in $x,y$ that is $\He(t,x,y)=\He(t,y,x)$.
\item The semigroup identity: for any $s\in(0,t)$
\begin{equation}\label{semigroup}
\He(t,x,y)=\int_{M}\He(s,x,z)\He(t-s,z,y)d\text{V}(z).
\end{equation}
\item For all $t>0$ and $x\in M$,
\begin{equation}
\int_{M}\He(t,x,y)d\text{V}(y)\leq 1.
\end{equation}
\end{enumerate}
If $M$ is the Euclidean space $\erre^n$ then, due to the homogeneity and isotropy of the Euclidean space, the heat kernel $\He^{n,0}(t,x,y)$ depends only on $t$ and $\rho(x,y)=\text{dist}(x,y)$, and is given by the classical formula
\begin{equation}
\He^{n,0}(t,\rho(x,y))=\frac{1}{(4\pi t)^{\frac{n}{2}}}e^{-\frac{\rho^2(x,y)}{4t}}\quad.
\end{equation}

A manifold $M$ satisfying for all $x\in M$ and all $t>0$

\begin{equation}
\int_{M}\He(t,x,y)d\text{V}(y)=1,
\end{equation}
is said to be stochastically complete. 

In the following proposition is proved that a complete submanifold of polynomial volume growth is stochastically complete

\begin{proposition}\label{completeness}Let $M^m$ be a $m$-dimensional complete minimal submanifold properly immersed in a Cartan-Hadamard ambient manifold. Suppose that $M^m$ is of polynomial volume growth, then  $M^m$ is stochastically complete  
\end{proposition}
\begin{proof}
Since $M^m$ has polynomial volume growth by equation (\ref{def-poly}), for any $o\in M$ and any $R\in \erre_+$ we have
\begin{equation}
\Vol(M^m\cap B_R^N(o))\leq \E \omega_mR^m.
\end{equation}
But since the  geodesic ball $B_R^{M^m}(o)$ of radius $R$ in $M^m$ is a subset of the extrinsic ball $M^m\cap B_R^N(o)$, one obtains that
\begin{equation}
\begin{aligned}
\int^\infty \frac{rdr}{\log\left(\Vol(B_r^{M^m}(o))\right)}&\geq \int^\infty \frac{rdr}{\log\left(\Vol(M^m\cap B_r^N(o))\right)}\\
&\geq \int^\infty \frac{rdr}{\log\left( \E \omega_mr^m\right)} =\infty.
\end{aligned}
\end{equation}

Hence, by \cite[theorem 9.1]{GriExp} $M^m$ is stochastically complete.
\end{proof}

Finally in order to conclude this preliminary section let us recall here the coarea formula
\begin{theorem}[Coarea formula, see \cite{Sakai,Chavel}]\label{coarea}
Let $f$ be a proper $C^\infty$ function defined on a Riemannian manifold $(M^n,g)$. Now we set
\begin{equation}
\begin{aligned}
\Omega_t:=\left\{p\in M;\, f(p)<t \right\},&\quad \text{V}_t:=\Vol(\Omega_t),\\
\Gamma_t:=\left\{p\in M;\, f(p)=t \right\},&\quad \text{A}_t:=\Vol_{n-1}(\Gamma_t).
\end{aligned}
\end{equation}
Then for an integrable function $u$ on $M^n$ the following hold:
\begin{enumerate}
\item Let $g_t$ be the induced metric on $\Gamma_t$ from $g$. Then
\begin{equation}
\int_{M^n} u\vert \nabla f\vert d\nu_g=\int_{-\infty}^\infty dt\int_{\Gamma_t}ud\nu_{g_t}.
\end{equation}
\item $t\to \text{V}_t$ is of class $C^\infty$ at a regular value $t$ of $f$ such that $\text{V}_t<+\infty$, and
\begin{equation}
\frac{d}{dt}\text{V}_t=\int_{\Gamma_t}\frac{1}{\vert\nabla f\vert}d\nu_{g_t}.
\end{equation}
\end{enumerate}
\end{theorem}

\section{Proof of the main theorem}
First of all, let us denote by $D_R(x)$ the extrinsic ball of radius $R$ cantered at $x$, i.e.,
$$
D_R(x):=M^m\cap B_R^N(x),
$$
therefore $\Q(R)$ is given by
$$
\Q(R)=\frac{\Vol(D_R(x))}{\omega_mR^m}.
$$
Note that $D_R(x)$ is the sublevel set of the extrinsic distance function $r_x$:
\begin{equation}
D_R(x)=\left\{p\in M^m;\, r_x(p)<R\right\}.
\end{equation}

Making use of the upper bounds for the heat kernel (equation \ref{upper}) and the semigroup property of the heat kernel (equation \ref{semigroup})
\begin{equation}
\begin{aligned}
1\geq \left(4\pi t\right)^\frac{m}{2}\He(t,x,x)&=\left(4\pi t\right)^\frac{m}{2}\int_{M^m}\He(t/2,x,y)^2d\text{V}(y)\\
&\geq \left(4\pi t\right)^\frac{m}{2}\int_{D_R(x)}\He(t/2,x,y)^2d\text{V}(y),
\end{aligned}
\end{equation}
for any extrinsic ball $D_R(x)$. Applying now the Cauchy–Schwarz inequality
\begin{equation}
\begin{aligned}
1\geq \left(4\pi t\right)^\frac{m}{2}\He(t,x,x)&\geq \left(4\pi t\right)^\frac{m}{2}\frac{\left(\int_{D_R(x)}\He(t/2,x,y)d\text{V}(y)\right)^2}{\Vol(D_R(x))},
\end{aligned}
\end{equation}
 Since by proposition \ref{completeness} $M^m$ is stochastically complete
\begin{equation}
\begin{aligned}
1\geq \left(4\pi t\right)^\frac{m}{2}\He(t,x,x)&\geq \left(4\pi t\right)^\frac{m}{2}\frac{\left(1-\int_{M^m\setminus D_R(x)}\He(t/2,x,y)d\text{V}(y)\right)^2}{\Vol(D_R(x))},
\end{aligned}
\end{equation}
Applying the polynomial volume growth property
\begin{equation}
\begin{aligned}
1\geq \left(4\pi t\right)^\frac{m}{2}\He(t,x,x)&\geq \left(4\pi t\right)^\frac{m}{2}\frac{\left(1-\int_{M^m\setminus D_R(x)}\He(t/2,x,y)d\text{V}(y)\right)^2}{\E \omega_mR^m},
\end{aligned}
\end{equation}
for all $R>0$. If we choose
\begin{equation}\label{Rt}
R=R_t:=\frac{\left(4\pi\right)^{\frac{1}{2}}}{\omega_m^{\frac{1}{m}}}t^{\frac{1}{2}}=2 \left[\frac{m}{2}\Gamma\left(\frac{m}{2}\right)\right]^{\frac{1}{m}}t^{\frac{1}{2}},
\end{equation}
we obtain

\begin{equation}\label{abans-subs}
\begin{aligned}
1\geq \left(4\pi t\right)^\frac{m}{2}\He(t,x,x)&\geq \frac{\left(1-\int_{M^m\setminus D_{R_t}(x)}\He(t/2,x,y)d\text{V}(y)\right)^2}{\E},
\end{aligned}
\end{equation}

We need now the following proposition

\begin{proposition}
Suppose that 
$$
\lim_{R\to \infty}\Q(R)=\E
$$
then 
\begin{equation}
\int_{M^m\setminus D_{R_t}(x)}\He(t/2,x,y)d\text{V}(y)\leq \E \left(C_m+\delta(t)\right) , 
\end{equation}
being $\delta$ a smooth function with $\delta\to 0$ when  $t\to \infty$.
\end{proposition}
\begin{proof}
By inequality (\ref{Mar})
\begin{equation}
\begin{aligned}
\int_{M^m\setminus D_{R_t}(x)}\He(t/2,x,y)d\text{V}(y)&\leq \int_{M^m\setminus D_{R_t}(x)}\He^{m,0}(t/2,r_x(y))d\text{V}(y)
\end{aligned}
\end{equation}
by coarea formula (theorem \ref{coarea})
\begin{equation}
\begin{aligned}
\int_{M^m\setminus D_{R_t}(x)}\He(t/2,x,y)d\text{V}(y)\leq & \int_{R_t}^\infty\int_{\partial D_S(x)}\frac{\He^{m,0}(t/2,r_x(y))}{\vert \nabla r_x\vert}d\text{V}_{s}(y)ds\\
\leq & \int_{R_t}^\infty\He^{m,0}(t/2,s)\left(\Vol(D_s(x)\right)'ds.
\end{aligned}
\end{equation}
The derivative $\frac{d}{dR}\Vol(D_R(o))=\left(\Vol(D_R)\right)'$ satisfies
\begin{equation}
\left(\Vol(D_R)\right)'=m\omega_m\Q(R)R^{m-1}+\omega_mR^m \Q(R)\left(\log(\Q(R)\right)'.
\end{equation}
Therefore,
\begin{equation}
\begin{aligned}
&\int_{{M^m\setminus D_{R_t}(x)}}\He(t/2,x,y)d\text{V}(y)\leq\\
&\frac{\omega_m}{(2\pi t)^{\frac{m}{2}}}\int_{R_t}^\infty e^{-\frac{s^2}{2t}}\left[ m\Q(s)s^{m-1}+s^m \Q(s)\left(\log(\Q(s)\right)'\right]ds\leq\\
&  \frac{\omega_m\E}{(2\pi t)^{\frac{m}{2}}}\int_{R_t}^\infty e^{-\frac{s^2}{2t}}\left[m s^{m-1}+s^m\left(\log(\Q(s)\right)'\right]ds \leq\\
& \frac{\omega_m\E}{(2\pi t)^{\frac{m}{2}}}\left[\int_{R_t}^\infty me^{-\frac{s^2}{2t}}s^{m-1}ds+\left(\sup_{s\in[0,\infty)}e^{-\frac{s^2}{2t}}s^m\right)\log\left(\frac{\E}{\Q(R_t)}\right)\right]=\\
 &\frac{\omega_m\E}{(2\pi)^{\frac{m}{2}}}\left[m2^{\frac{m}{2}-1}\Gamma(\frac{m}{2},\frac{{R_t}^2}{2t})+e^{-\frac{m}{2}}m^{\frac{m}{2}}\log\left(\frac{\E}{\Q(R_t)}\right)\right] .
\end{aligned}
\end{equation}
Taking into account the definition of $R_t$ (equation (\ref{Rt})) and that $\omega_m=\frac{2\pi^{\frac{m}{2}}}{m\Gamma(\frac{m}{2})}$,

\begin{equation}
\begin{aligned}
&\int_{M^m\setminus D_{R_t}(x)}\He(t/2,x,y)d\text{V}(y)\leq  \E \left[C_m+\delta(t)\right] ,
\end{aligned}
\end{equation}
where
$$
\delta(t):=\frac{e^{-\frac{m}{2}}(\frac{m}{2})^{\frac{m}{2}-1}}{\Gamma(\frac{m}{2})}\log\left(\frac{\E}{\Q\left(2 \left(\frac{m}{2}\Gamma\left(\frac{m}{2}\right)\right)^{\frac{1}{m}}t^{\frac{1}{2}}\right)}\right).
$$
Making use that $\Q(s)=\E$ when $s\to\infty$ the proposition is proven.
\end{proof}

Hence for $t$ large enough  we can apply the above proposition in equation (\ref{abans-subs}) 
\begin{equation}
\begin{aligned}
1\geq \left(4\pi t\right)^\frac{m}{2}\He(t,x,x)&\geq \frac{\left[1- \E \left(C_m+\delta(t)\right)\right]^2}{\E}.
\end{aligned}
\end{equation}
Therefore, taking limits the theorem follows.

%%%%%%%%%%%%%%%%%%%%%%%%%%%%%%%%%%%%%%%%%%%%%%%%%%%%5
% BIBLIOGRAPHY 
%%%%%%%%%%%%%%%%%%%%%%%%%%

\def\cprime{$'$} \def\cprime{$'$} \def\cprime{$'$} \def\cprime{$'$}
  \def\cprime{$'$}
\providecommand{\bysame}{\leavevmode\hbox to3em{\hrulefill}\thinspace}
\providecommand{\MR}{\relax\ifhmode\unskip\space\fi MR }
% \MRhref is called by the amsart/book/proc definition of \MR.
\providecommand{\MRhref}[2]{%
  \href{http://www.ams.org/mathscinet-getitem?mr=#1}{#2}
}
\providecommand{\href}[2]{#2}

\end{document}